\documentclass[a4paper,11pt]{amsart}
\usepackage{amsmath,amsfonts,amsthm,amssymb,enumerate}
\usepackage{graphicx}

\numberwithin{equation}{section}

\newtheorem{theorem}{Theorem}[section]

\newtheorem{proposition}[theorem]{Proposition}

\newtheorem{lemma}[theorem]{Lemma}
\newtheorem{remark}[theorem]{Remark}

{
\theoremstyle{definition}
\newtheorem{definition}[theorem]{Definition}

}

\newenvironment{rem}{\begin{remark}\rm}{\end{remark}}

\title{Showing distinctness of surface links by taking 2-dimensional braids}
\author{Inasa Nakamura}
\address{
Institute for Biology and Mathematics of Dynamical Cell Processes (iBMath), Interdisciplinary Center for Mathematical Sciences, Graduate School of Mathematical Sciences, The University of Tokyo\newline
3-8-1 Komaba, Tokyo 153-8914, Japan}
\email{inasa@ms.u-tokyo.ac.jp}

\begin{document}
\maketitle

\begin{abstract}
For an oriented surface link $S$, 
we can take a satellite construction called a 2-dimensional braid over $S$, which is a surface link in the form of a covering over $S$. 
We demonstrate that 2-dimensional braids over surface links are useful for showing the distinctness of surface links. 
We investigate non-trivial examples of surface links with free abelian group of rank two, concluding that their link types are infinitely many. 
\end{abstract}

\section{Introduction}

A {\it surface link} is the image of a smooth embedding of a closed surface into Euclidean space $\mathbb{R}^4$.  Two surface links are {\it equivalent} if there is an orientation-preserving self-diffeomorphism of $\mathbb{R}^4$ carrying one to the other. 
In this paper, we assume that surface links are oriented. 
In \cite{N4}, we investigated a satellite construction called a 2-dimensional braid over an oriented surface link, and introduced its graphical presentation called an $m$-chart on a surface diagram. 
A 2-dimensional braid over a surface link $S$ is a surface link in the form of a covering over $S$, and can be regarded as an analog to a double of a classical link. 
One of expected applications of the notion of a 2-dimensional braid is that it will provide us with a method for showing the distinctness of surface links. 
The aim of this paper is to demonstrate such use for 2-dimensional braids.

\begin{sloppypar}
Our main theorem is as follows. Let $k$ be a positive integer. Let $\sigma_1, \sigma_2, \ldots, \sigma_k$ be the standard generators of the $(k+1)$-braid group. Let $X_k=\sigma_1^2 \sigma_2 \sigma_3 \cdots \sigma_k$ where $X_1=\sigma_1^2$, and let $\Delta$ be a $(k+1)$-braid with a positive half twist. 
Let $S_k=\mathcal{S}_{k+1}(X_k, \Delta^2)$, a $T^2$-link called a torus-covering $T^2$-link determined from $(k+1)$-braids $X_k$ and $\Delta^2$, and we take the first (respectively second) component of $S_k$ as the one determined from the first (respectively second) strand of $X_k$; see Section \ref{sec:5} for the construction. Here, a {\it $T^2$-link} is a surface link each of whose components is of genus one.   
\begin{theorem}\label{mainthm}
Abelian $T^2$-links of rank two, $S_k$ and $S_l$, are not equivalent for distinct positive integers $k$ and $l$. Thus, the link types of abelian $T^2$-links of rank two are infinitely many. 
\end{theorem}
An {\it abelian surface link} of rank $n$ is a surface link whose link group is a free abelian group of rank $n$ \cite{I-N}; note that $n$ is the number of the components. 
We remark that our abelian $T^2$-links of rank two cannot be distinguished by using link groups, and that by a homological argument we cannot show that their link types are infinitely many, but only that there are two link types; see Section \ref{sec:5-3}.
Our abelian $T^2$-link $S_k$ of rank two is a sublink of the surface link given in \cite{I-N}, where we gave examples of abelian $T^2$-links of rank four, and we showed that their link types are infinitely many by calculations of triple linking numbers (see also Remark \ref{rem0214}). Triple linking numbers are integer-valued invariants of surface links with at least three components; so we cannot use them for our case without a device. 
In order to overcome this situation, we take a 2-dimensional braid over $S_k$ such that each component of $S_k$ is split into two components. Then it has four components, and we can calculate triple linking numbers. A 2-dimensional braid over a surface link is obtained from the \lq\lq standard" 2-dimensional braid by addition of braiding information. Unfortunately, if we consider the standard 2-dimensional braid, then the triple linking is trivial (Proposition \ref{prop:tlk}). However, addition of braiding information makes a 2-dimensional braid with non-trivial triple linking, and enables us to show that $S_k$ and $S_l$ are not equivalent for distinct positive integers $k$ and $l$. As a similar result, we refer to Suciu's paper \cite{Suciu} where it is shown that there are infinitely many ribbon 2-knots in $S^4$ with knot group the trefoil knot group. 
\end{sloppypar}

The paper is organized as follows. In Section \ref{sec:5}, we review torus-covering links and explain our example $S_k$, and we review triple linking numbers of torus-covering links. In Section \ref{sec:2}, we review the notion of a 2-dimensional braid over a surface link. 
In Section \ref{sec:3}, we review that a 2-dimensional braid of degree $m$ over a surface link is presented by a finite graph called an $m$-chart on a surface diagram, and that 2-dimensional braids of degree $m$ are equivalent if their surface diagrams with $m$-charts are related by local moves called Roseman moves. 
In Section \ref{sec:4}, we show Proposition \ref{prop:tlk}. 
In Section \ref{sec:6}, we calculate triple linking numbers of a certain 2-dimensional braid over $S_k$ and prove Theorem \ref{mainthm}.

\section{Abelian $T^2$-links of rank two}\label{sec:5}
Our example $S_k$ given in Theorem \ref{mainthm} is a surface link called a torus-covering link. 
In this section, we review torus-covering $T^2$-links; see \cite{N1} for details. 
We briefly observe that $S_k$ is an abelian surface link of rank two, and that we cannot show that the link types of our examples are infinitely many by using a homological argument. 
Further, we review a formula for the triple linking numbers of torus-covering links \cite{I-N}.

\subsection{Torus-covering links}\label{sec:5-1}

Let $T$ be a standard torus in $\mathbb{R}^4$, the boundary of an unknotted (standardly embedded) solid torus in $\mathbb{R}^3 \times \{0\} \subset \mathbb{R}^4$. 

\begin{definition} \label{Def2-1} 
A {\it torus-covering $T^2$-link} $S$ is a surface link in the form of a 2-dimensional braid over the standard torus $T$, i.e. $S$ is a 
$T^{2}$-link in $\mathbb{R}^4$ 
such that $S$ is contained in a tubular neighborhood $N(T)$ and $\pi |_{S} \,:\, S \to T$ is an unbranched covering map, where $\pi \,:\, N(T) \to T$ is the natural projection. 
\end{definition} 

Let $S$ be a torus-covering $T^{2}$-link. Fix a base point $x_0 =(x'_{0},x''_{0})$ of $T = S^{1} \times S^{1}$. Take two simple closed curves on $T$, $\mathbf{m} = \partial B^{2} \times \{x''_0\}$ and $\mathbf{l}= \{x'_{0}\}\times S^{1}$. Recall that $T$ is embedded as $T=\partial (B^{2} \times S^{1}) \subset \mathbb{R}^{3} \times \{0\} \subset \mathbb{R}^{4}$. Let us consider the intersections $S \cap \pi^{-1}(\mathbf{m}) \subset B^2 \times \mathbf{m}$ and $S \cap \pi^{-1}(\mathbf{l}) \subset B^2 \times \mathbf{l}$. They are regarded as closed $m$-braids in the 3-dimensional solid tori, where $m$ is the degree of the covering map $\pi |_{S} \, :\, S  \rightarrow T$. Cutting open the solid tori along the 2-disk $\pi^{-1}(x_0)= B^2 \times \{x_{0}\}$, we obtain two $m$-braids $a$ and $b$.
The assumption that $\pi|_{S}$ is an unbranched covering implies that $a$ and $b$ commute. We call the commutative braids $(a, b)$ the {\it basis braids} of $S$. Conversely, starting from a pair of commutative $m$-braids $(a, b)$, we can uniquely construct a torus-covering $T^2$-link with basis braids $(a, b)$ \cite[Lemma 2.8]{N1}.
For commutative $m$-braids $a$ and $b$, we denote by $\mathcal{S}_{m}(a,b)$ the torus-covering $T^2$-link with basis braids $(a, b)$.

\subsection{Our abelian $T^2$-links of rank two}\label{sec:5-3}
We can check that our example $S_k=\mathcal{S}_{k+1}(X_k, \Delta^2)$ is an abelian surface link, as follows. 
The link group of a torus-covering link $\mathcal{S}_m(a,b)$ is a quotient group of the classical link group of the closure of $a$  such that the abelianization is a free abelian group \cite[Proposition 3.1]{N1}. Since the the link group of the closure of $X_k$, a Hopf link, is a free abelian group of rank two, so is the link group of $S_k$.

We remark that by a homological argument we cannot show that our examples are infinitely many, but only that there are two link types. Let us consider the one-point compactification of $\mathbb{R}^4$, and regard that $S_k$ is in the Euclidean 4-sphere $S^4$. Recall that we take the first (respectively second) component of $S_k$ as the one determined from the first (respectively second) strand of $X_k$, and let us denote by $F_1$ (respectively $F_2$) the first (respectively second) component of $S_k$. 
Then, by Alexander's duality, we see that $H_2(S^4-F_1; \mathbb{Z})\cong H_1(F_1; \mathbb{Z})$, hence $[F_2]=\mu+k\lambda\in H_2(S^4-F_1; \mathbb{Z})$, where $(\mu, \lambda)$ is a preferred basis of $H_1(F_1; \mathbb{Z}) \cong H_2(S^4-F_1; \mathbb{Z})$ represented by a meridian and a preferred longitude of $F_1$. 
Similarly, let us denote by $F_1^\prime$ (respectively $F_2^\prime$) the first (respectively second) component of $S_l$. Then we can see that $[F_2^\prime]=\mu^\prime+l\lambda^\prime\in H_2(S^4-F_1^\prime; \mathbb{Z})$, where $(\mu^\prime, \lambda^\prime)$ is a preferred basis of $H_1(F_1^\prime; \mathbb{Z}) \cong H_2(S^4-F_1^\prime; \mathbb{Z})$  represented by a meridian and a preferred longitude of $F_1^\prime$. 
Now, standardly embedded tori $F_1$ and $F_1^\prime$ are related by an orientation-preserving self-diffeomorphism of $S^4$ if and only if $\begin{pmatrix} \mu^\prime \\ \lambda^\prime \end{pmatrix}=A\begin{pmatrix} \mu \\ \lambda \end{pmatrix}$ for $A=\begin{pmatrix} \alpha & \beta\\
\delta & \gamma \end{pmatrix} \in GL_+(2; \mathbb{Z})$ such that $\alpha+\beta+\gamma+\delta\equiv 0 \pmod{2}$ \cite{Montesinos}, which implies that $[F_2]=[F_2^\prime] \in  H_2(S^4-F_1; \mathbb{Z})$ if and only if $k\equiv l \pmod{2}$ (see \cite{Iwase}). 

\begin{rem}\label{rem0211}
The abelian surface link $S_1$, i.e. $\mathcal{S}_2(\sigma_1^2, \sigma_1^2)$, is the twisted Hopf 2-link we will mention in the proof of Proposition \ref{prop:tlk}; see also \cite{CKSS01}. 
\end{rem}

\begin{rem}\label{rem0214}
It is known \cite[Theorem 6.3.1--Exercise 6.3.3]{Kawauchi} that for classical links, the rank of an abelian link is at most two, and, for abelian links of rank two, there are exactly two link types; a positive Hopf link and a negative Hopf link. 
\end{rem}

\begin{rem}
Put $T_m=\mathcal{S}_{k+1}(X_k, X_k^m)$ for an integer $m$. It is known (\cite{Boyle}, see also \cite{Iwase, N1}) that $T_m$ and $T_n$ are equivalent for $m \equiv n \pmod{2}$. Fix the first component of $T_m$ in the form of the standard torus. By a homological argument as in this section, we see that $T_{m}$ cannot be taken to $T_{n}$ for $n \neq m$ by an orientation-preserving self-diffeomorphism of $\mathbb{R}^4$ relative to the first component.
\end{rem}

\subsection{Triple linking numbers of torus-covering links}\label{sec:5-2}

The triple linking number of a surface link $S$ is defined as follows \cite[Definition 9.1]{CJKLS}. 
For the $i$th, $j$th, and $k$th components $F_i, F_j,F_k$ of $S$ with $i \neq j$ and $j \neq k$, the {\it triple linking number} $\mathrm{Tlk}_{i,j,k}(S)$ of the $i$th, $j$th, and $k$th components of $S$ is the total number of positive triple points minus the total number of negative triple points of a surface diagram of $S$ such that the top, middle, and bottom sheet are from $F_i, F_j$, and $F_k$, respectively. Triple linking number is a link bordism invariant \cite{CKS, CKSS01,San, San2}; for other properties, see \cite{CJKLS, CKS}. Triple linking numbers are useful for showing the distinctness of surface links with at least three components \cite{I-N, N2, N3}.

By \cite{I-N}, we have a formula for the triple linking numbers of a torus-covering $T^2$-link $\mathcal{S}_{m}(a,b)$.

We use the notations given in \cite{I-N}. 
For a torus-covering $T^2$-link $\mathcal{S}_{m}(a,b)$, let ${A}_{i}$ be the components of the closure of $a$ which are from the $i$th component of $\mathcal{S}_{m}(a,b)$. Take one of the connected components of ${A}_{i}$ and denote it by $A^1_{i}$. 
We define $\mathrm{lk}^{a}_{i,j}$ by the classical linking number
\[ \mathrm{lk}_{i,j}^{a} = \mathrm{lk} ( A_{i}^1, {A}_{j}), \]
where we regard ${A}_{i}^1$ and $A_j$ as oriented links in $\mathbb{R}^{3}$. The notation $\mathrm{lk}^{b}_{i,j}$ for the other basis braid is defined similarly.
Note that $\mathrm{lk}^{a}_{i,j}$ does not depend on a choice of a connected component $A_{i}^1$ \cite[Remark 5.5]{I-N}, and note that $\mathrm{lk}^{a}_{i,j}$ is not always symmetric, i.e. $\mathrm{lk}^{a}_{i,j}$ is not always equal to $\mathrm{lk}^{a}_{j,i}$.

For a torus-covering $T^{2}$-link, the triple linking number of the $i$th, $j$th and $k$th components is given by 
\begin{equation}\label{eq:tlk}
 \mathrm{Tlk}_{i,j,k}(\mathcal{S}_{m}(a,b)) = -\mathrm{lk}^{a}_{j,i}\mathrm{lk}^{b}_{j,k} + \mathrm{lk}^{a}_{j,k}\mathrm{lk}^{b}_{j,i}, 
\end{equation}
where $i \neq k$ and $j \neq k$ \cite[Theorem 5.4 and Remark 5.7]{I-N}. 

\section{Two-dimensional braids over a surface link}\label{sec:2}
A 2-dimensional braid, which is also called a simple braided surface, over a 2-disk, is an analogous notion of a classical braid \cite{Kamada92,Kamada02,Rudolph}. 
We can modify this notion to a 2-dimensional braid over a closed surface \cite{N1}, and further to a 2-dimensional braid over a surface link \cite[Section 2.4.2]{CKS}, \cite{N4}. 

In this section, we review the notion of a 2-dimensional braid over a surface link \cite{N4}.

\subsection{Two-dimensional braids over a surface link}
We use 2-dimensional braids without branch points over a closed surface, so our definition here is restricted to such surfaces; see \cite{N1,N4} for the definition which allows branch points.

Let $\Sigma$ be a closed surface, let $B^2$ be a 2-disk, and let $m$ be a positive integer. 
\begin{definition}
A closed surface $\widetilde{\Sigma}$ embedded in $B^2 \times \Sigma$ is called a {\it 2-dimensional braid over $\Sigma$ of degree $m$} if  
the restriction $\pi |_{\widetilde{\Sigma}} \,:\, \widetilde{\Sigma} \rightarrow \Sigma$ is an unbranched covering map of degree $m$, where $\pi \,:\, B^2 \times \Sigma \to \Sigma$ is the natural projection. 

Take a base point $x_0$ of $\Sigma$. 
Two 2-dimensional braids over $\Sigma$ of degree $m$ are {\it equivalent} if there is a fiber-preserving ambient isotopy of $B^2 \times \Sigma$ rel $\pi^{-1}(x_0)$ which carries one to the other. 

\end{definition}

A surface link is said to be {\it of type $\Sigma$} when it is the image of an embedding of $\Sigma$. 
Let $S$ be a surface link of type $\Sigma$, and let $N(S)$ be a tubular neighborhood of $S$ in $\mathbb{R}^4$. 

\begin{definition}\label{def:2-braid}
A {\it 2-dimensional braid} $\widetilde{S}$ {\it over $S$} is the image of a 2-dimensional braid over $\Sigma$ in $B^2 \times \Sigma$ by an embedding $B^2 \times \Sigma \to \mathbb{R}^4$ which identifies $N(S)$ with $B^2 \times \Sigma$ as a $B^2$-bundle over a surface.
We define the {\it degree} of $\widetilde{S}$ as that of $S$.

Two 2-dimensional braids $\widetilde{S}$ and $\widetilde{S^\prime}$ over surface links $S$ and $S^\prime$ are {\it equivalent} if there is an ambient isotopy of $\mathbb{R}^4$ carrying $\widetilde{S}$ to $\widetilde{S^\prime}$ and $N(S)=B^2 \times S$ to $N(S^\prime)=B^2 \times S^\prime$ as a $B^2$-bundle over a surface. 

\end{definition}
 
Equivalent 2-dimensional braids over surface links are also equivalent as surface links. 
A 2-dimensional braid $\widetilde{S}$ over $S$ is a specific satellite with companion $S$; see \cite[Section 2.4.2]{CKS}, see also \cite[Chapter 1]{Lickorish}.

\subsection{Standard 2-dimensional braids}\label{sec:2-2}
In this section, we define the standard 2-dimensional braid over a surface link $S$. Using this notion, we will explain in the next section that a 2-dimensional braid is presented by a finite graph called an $m$-chart on a surface diagram $D$ of $S$. The standard 2-dimensional braid over $S$ is the 2-dimensional braid presented by an empty $m$-chart on $D$ \cite{N4}.

First we will review a surface diagram of a surface link $S$; see \cite{CKS}. For a projection $p \,:\, \mathbb{R}^4 \to \mathbb{R}^3$, the closure of the self-intersection set of $p(S)$ is called the singularity set. Let $p$ be a generic projection, i.e. the singularity set of the image $p(S)$ consists of double points, isolated triple points, and isolated branch points; see Figure \ref{0215-1}. The closure of the singularity set forms a union of immersed arcs and loops, called double point curves. Triple points (respectively branch points) form the intersection points (respectively end points) of the double point curves. A {\it surface diagram} of $S$ is the image $p(S)$ equipped with over/under information along each double point curve with respect to the projection direction. 
 
\begin{figure}
\begin{center}
 \includegraphics*{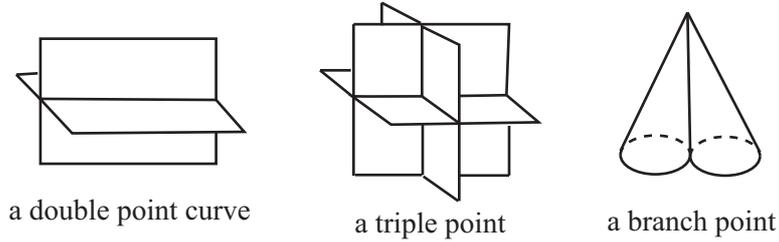}
\end{center}
  \caption{The singularity of a surface diagram.}
  \label{0215-1}
  \end{figure}

We define the $2m$-braid $\widetilde{\sigma_1}$ obtained from a $2$-braid $\sigma_1$, as follows. For the proof of Theorem \ref{mainthm}, here we define the $mn$-braid $\widetilde{b}$ obtained from an $n$-braid $b$. 
Let $Q_m$ be $m$ interior points of $B^2$. 
For a standard generator $\sigma_i$ of an $n$-braid, let $\widetilde{\sigma_i}$ be the $mn$-braid obtained from $\sigma_i$ in such a way that it is in the form of a $Q_m$-bundle over $\sigma_i$ and it is obtained from $\sigma_i$ by splitting each strand into a bundle of $m$ parallel strands with a negative half twist at the initial points of each bundle; see Figure \ref{fig2014-0210-01}. The map taking $\sigma_i$ to $\widetilde{\sigma_i}$ determines a homomorphism from the $n$-braid group to the $mn$-braid group. For an $n$-braid $b$, let $\widetilde{b}$ denote the image of $b$ by this homomorphism. 

\begin{figure}
\begin{center}
 \includegraphics*{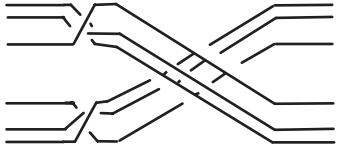}
\end{center}
  \caption{The $2m$-braid $\widetilde{\sigma_1}$.}
  \label{fig2014-0210-01}
  \end{figure}

\begin{definition}
Let $S$ be a surface link. 
A surface diagram $D$ of $S$ consists of the following local parts: around (1) a regular point i.e. a nonsingular point, (2) a double point curve, (3) a triple point, and (4) a branch point. The case (1) is presented by an embedded 2-disk $B^2$ with no singularity, and the case (2) is presented as the product of a $2$-braid $\sigma_1$ and an interval $I$. 
 
We define the {\it standard 2-dimensional braid over $S$} locally for such local parts of $D$ as follows: for (1), it is $m$ parallel copies of $B^2$, and for (2), it is the product of the $2m$-braid $\widetilde{\sigma_1}$ and $I$. Then, for the other cases (3) and (4), the standard 2-dimensional braid is naturally defined \cite[Definition 5.1 and Proposition 5.2]{N4}. 
\end{definition}

\section{Chart presentation of 2-dimensional braids \newline
and Roseman moves}\label{sec:3}
In this section, we review the following. A 2-dimensional braid of degree $m$ over a surface link $S$ is presented by a finite graph called an $m$-chart on a surface diagram $D$ of $S$ \cite{N4}. For two 2-dimensional braids of degree $m$, they are equivalent if their surface diagrams with $m$-charts are related by a finite sequence of local moves called Roseman moves \cite{N4}. 

\subsection{Chart presentation of 2-dimensional braids over a surface link} \label{sec:3-1}
 The graphical method called an $m$-chart on a 2-disk was introduced to present a simple surface braid which is a 2-dimensional braid over a 2-disk with trivial boundary condition \cite{Kamada92, Kamada02}. By regarding an $m$-chart on a 2-disk as drawn on a 2-sphere $S^2$, it presents a 2-dimensional braid over $S^2$ \cite{Kamada92, Kamada02, N1}. 
This notion can be modified to an $m$-chart on a closed surface \cite{N1}, and further to an $m$-chart on a surface diagram $D$ of a surface link $S$ \cite{N4}. A 2-dimensional braid over $S$ is presented by an $m$-chart on $D$ \cite{N4}. 
\\

In this paper, we treat $2$-charts with vertices of degree $2$. We will just review the graphical form of an $m$-chart of a 2-dimensional braid over a surface link. See \cite{N4} for details.

Let $\widetilde{S}$ be a 2-dimensional braid over a surface link $S$. Let $D$ be a surface diagram of $S$ by a projection $p:\mathbb{R}^4 \to \mathbb{R}^3$ which is generic with respect to both $S$ and $\widetilde{S}$. 
We can assume that the singularity set of the surface diagram 
$p(\widetilde{S})$ is the union of the singularity set of the diagram of the standard 2-dimensional braid over $S$ and some finite graph $\Gamma$ \cite[Theorem 5.5]{N4}. 
Project $\Gamma$ to $D$ by the projection $p(N(S))=B^2 \times D \to D$. Then we obtain a finite graph on the surface diagram $D$. An $m$-chart on a surface diagram $D$ is such a finite graph equipped with certain additional information of orientations and labels assigned to the edges, where $m$ is the degree of the 2-dimensional braid. Owing to the additional information, we can regain the original 2-dimensional braid from the $m$-chart on $D$ \cite{N4} (see also \cite{Kamada02}).

 We can define an $m$-chart on $D$ in graphical terms, where the labels of edges are from $1$ to $m-1$; see \cite[Definitions 5.3 and 5.4]{N4}. 
Around a double point curve, an $m$-chart is as in Figure \ref{fig:0417-01}, with a vertex of degree $2$. 
A 2-dimensional braid over $S$ is presented by an $m$-chart on $D$ \cite[Theorem 5.5]{N4}. 

 \begin{figure}[ht]
 \centering\includegraphics*{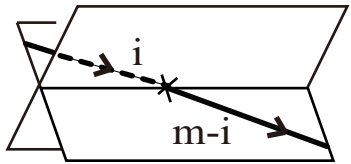}
\caption{An $m$-chart around a double point curve, where $i \in \{1,\ldots,m-1\}$. For simplicity, we omit the over/under information of each sheet.}
\label{fig:0417-01}
 \end{figure}

\subsection{Roseman moves}
Roseman moves are local moves of surface diagrams as illustrated in Figure \ref{fig:0417-03}. 
It is known \cite{Roseman} that two surface links are equivalent if and only if their surface diagrams are related by a finite sequence of Roseman moves and ambient isotopies of the diagrams in $\mathbb{R}^3$.
In \cite{N4}, we introduced the notion of Roseman moves for surface diagrams with $m$-charts. 

An $m$-chart is said to be {\it empty} if it is an empty graph. 
\begin{figure}
\centering
\includegraphics*{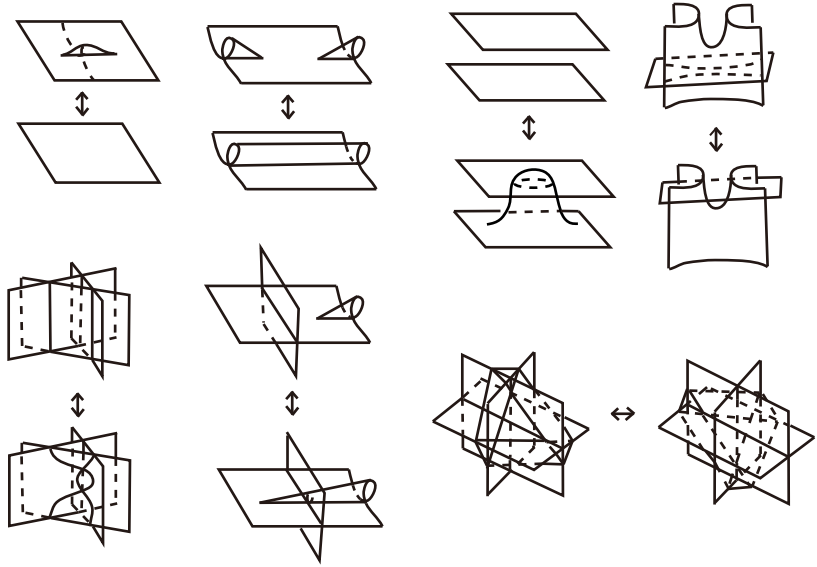}
\caption{Roseman moves. For simplicity, we omit the over/under information of each sheet. }
\label{fig:0417-03}
 \end{figure}

\begin{definition}
We define {\it Roseman moves for surface diagrams with $m$-charts} by the local moves as illustrated in Figures \ref{fig:0417-03} and \ref{fig:0417-04}, where we regard the diagrams in Figure \ref{fig:0417-03} as equipped with empty $m$-charts. 
\end{definition}

\begin{figure}
 \includegraphics*{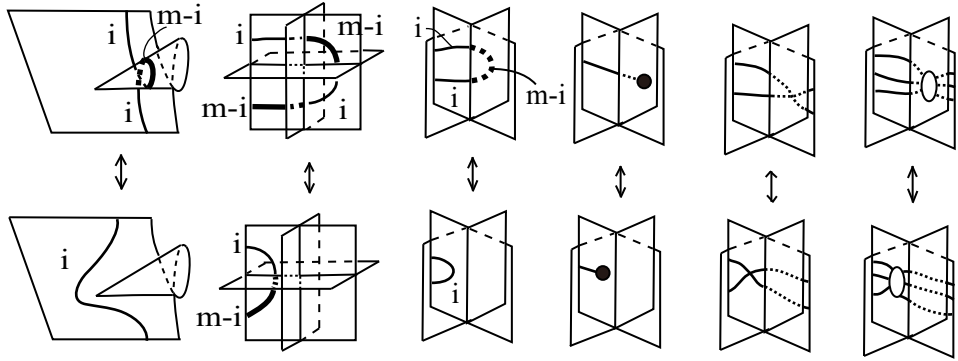}
\caption{Roseman moves for surface diagrams with $m$-charts, where $i \in \{1,\ldots,m-1\}$. For simplicity, we omit the over/under information of each sheet, and orientations and labels of edges of $m$-charts. }
\label{fig:0417-04}
 \end{figure}%
 
Roseman moves for surface diagrams with $m$-charts as illustrated in Figures \ref{fig:0417-03} and \ref{fig:0417-04} are well-defined, i.e. for each pair of Roseman moves, the $m$-charts on the indicated diagrams present equivalent 2-dimensional braids \cite[Theorem 6.2]{N4}.

\section{Triple linking numbers of standard 2-dimensional braids}\label{sec:4}

Recall the triple linking numbers (see Section \ref{sec:5-2}). 
We will say that a surface link $S$ has {\it trivial} triple linking if every triple linking number of $S$ is zero or $S$ consists of less than three components. 

\begin{proposition}\label{prop:tlk}
For the standard 2-dimensional braid $\tilde{S}$ over a surface link $S$, if $S$ has trivial triple linking, then so does $\tilde{S}$. 
\end{proposition}

\begin{proof}
Assume that $S$ has trivial triple linking. 
It is known \cite{CKSS01} that the link bordism class of a surface link is determined from triple linking numbers and another kind of link bordism invariants called double linking numbers, and a surface link with trivial triple linking is link bordant to a split union of a finite number of trivial spheres and surface links called twisted Hopf 2-links, which has a surface diagram with no triple points (see also Remark \ref{rem0211}). 
 Hence $S$ is link bordant to a surface link $S^\prime$ whose surface diagram has no triple points. By the well-definedness of Roseman moves, $\widetilde{S}$ is link bordant to the standard 2-dimensional braid $\widetilde{S^\prime}$ over $S^\prime$. Since the surface diagram of a standard 2-dimensional braid has triple points only around triple points of the companion surface \cite{N4}, the surface diagram of $\widetilde{S^\prime}$ has no triple points. Thus $\widetilde{S}$ is link bordant to a surface link with no triple points, which implies that $\widetilde{S}$ has trivial triple linking.
\end{proof}

\section{Proof of Theorem \ref{mainthm}}\label{sec:6}

In this section, we will consider a 2-dimensional braid $\widetilde{S}$ over a surface link $S$ presented by a $2$-chart consisting of a finite number of loops on a surface diagram of $S$. Here, a {\it loop} is a union of edges connected by vertices of degree $2$ as in Figure \ref{fig:0417-01}. 
In our case of $2$-charts, the edges are labeled by $1$ and the orientations are coherent around a vertex of degree $2$, so we can ignore the label information, and we regard the $2$-chart on a surface diagram of $S$ as oriented loops. Further, we consider that the loops are on $S$ itself. 
By the well-definedness of Roseman moves, a 2-dimensional braid presented by a $2$-chart $\Gamma$ on $S$ is equivalent to the 2-dimensional braid presented by a $2$-chart $f(\Gamma)$ on $f(S)$ for an orientation-preserving self-diffeomorphism $f$ of $\mathbb{R}^4$. 
 
For a component $F$ of a torus-covering $T^2$-link, we take a preferred basis of $H_1(F; \mathbb{Z})$ represented by a pair of simple closed curves $(\mu, \lambda)$ such that $\mu$ (respectively $\lambda$) is a connected component of $F \cap \pi^{-1}(\mathbf{m})$ (respectively $F \cap \pi^{-1}(\mathbf{l})$). Recall that $\pi: N(T) \to T$ is the natural projection for a standard torus $T$, and $\mathbf{m}$ and $\mathbf{l}$ are simple closed curves on $T$ given in Section \ref{sec:5-1}. We will use the same notation $(\mu, \lambda)$ for the preferred basis, and we call a simple closed curve in the homology class $\mu$ (respectively $\lambda$) a {\it meridian} (respectively a {\it preferred longitude}) of $F$. 
For a $2$-chart $\Gamma$ on $F$ consisting of loops, we can assume that 
the intersections of the chart loops of $\Gamma$ with a meridian $\mu$ and a preferred longitude $\lambda$ of $F$ are transverse. We assign each intersection point the sign $+1$ (respectively $-1$) when it presents a positive (respectively negative) crossing, and we denote by $I(\mu, \Gamma)$ (respectively $I(\lambda, \Gamma)$) the sum of the signs of the intersection points of $\Gamma$ with $\mu$ (respectively $\lambda$); note that we can assume that the chart loops are parallel by using local moves of charts called CI-moves of type (1) \cite{Kamada02}, and $I(\mu, \Gamma)$ and $I(\lambda, \Gamma)$) are well-defined for the homology classes $\mu$ and $\lambda$. 


For the torus-covering $T^2$-link $S$ and its 2-dimensional braid $\widetilde{S}$ treated in this section, we take the first (respectively second) component of $S$ as the one determined from the first (respectively second) strand of each basis braid of $S$, and we take the $i$th component of $\widetilde{S}$ as the one determined from the $i$th strand of each basis braid of $\widetilde{S}$ for $i=1,2,3,4$. 

For the proof of Theorem \ref{mainthm}, we calculate the triple linking numbers of a 2-dimensional braid of degree $2$ over $S_k$ in Theorem \ref{mainthm}.  

\begin{lemma}\label{lemma2}
For the torus-covering $T^2$-link $S_k$ for a positive integer $k$ in Theorem \ref{mainthm}, let us consider a 2-dimensional braid of degree $2$ over $S_k$, denoted by $\widetilde{S_k}$, which is presented by a $2$-chart $\Gamma$ consisting of loops on $S_k$ such that it consists of $4$ components. 
Then $\mathrm{Tlk}_{i,j,3}(\widetilde{S_k})=\mathrm{Tlk}_{i,j,4}(\widetilde{S_k})$ for $(i,j)=(1,2)$ or $(2,1)$, and 
$\mathrm{Tlk}_{i,j,1}(\widetilde{S_k})=\mathrm{Tlk}_{i,j,2}(\widetilde{S_k})$ for $(i,j)=(3,4)$ or $(4,3)$. 
\end{lemma}

\begin{proof}
The 2-dimensional braid $\widetilde{S_k}$ is also a torus-covering $T^2$-link. 
We denote by $(a, b)$ the basis braids presenting $\widetilde{S_k}$. 
Since $\mathrm{lk}_{j,3}^c=\mathrm{lk}_{j,4}^c$ for $j=2,1$, and 
 $\mathrm{lk}_{j,1}^c=\mathrm{lk}_{j,2}^c$ for $j=4,3$ $(c=a,b)$, by (\ref{eq:tlk}) we have the result. 
\end{proof}

\begin{lemma}\label{lemma1}
For the torus-covering $T^2$-link $S_k$, let us denote by $F_1$ (respectively $F_2$) the first (respectively second) component of $S_k$, and let $(\mu_i, \lambda_i)$ be a preferred basis of $H_1(F_i; \mathbb{Z})$ $(i=1,2)$. Let us consider a 2-dimensional braid $\widetilde{S_k}$ as in Lemma \ref{lemma2}, such that $I(\mu_i, \Gamma)=2p_i$ and  $I(\lambda_i, \Gamma)=2q_i$, for integers $p_i$ and $q_i$ $(i=1,2)$. 
Then 
$\mathrm{Tlk}_{1,2,3}(\widetilde{S_k})=-k p_1 +q_1$ and 
$\mathrm{Tlk}_{2,3,4}(\widetilde{S_k})=-p_2+q_2$. 
\end{lemma}

Note that $\widetilde{S_k}$ consists of 4 components if and only if $I(\mu_i, \Gamma)$ and $I(\lambda_i, \Gamma)$ $(i=1,2)$ are even, since these conditions are equivalent to the condition that $\widetilde{S_k}\cap \pi_i^{-1}(\mu)$ and $\widetilde{S_k} \cap \pi_i^{-1}(\lambda)$ $(i=1,2)$ are closed pure braids, where $\pi_i: N(F_i) \to F_i$ is the natural projection. 

\begin{proof}
The 2-dimensional braid $\widetilde{S_k}$ is also a torus-covering $T^2$-link. 
We denote by $(a, b)$ the basis braids presenting $\widetilde{S_k}$. 
We use the notation given in Section \ref{sec:2-2}, taking $m=2$ and $n=k+1$. 
Then, $\mathrm{lk}_{2,1}^a$ is determined from the linking number coming from the linking consisting of $I(\mu_1, \Gamma)$ crossings and $\widetilde{X_k}$, that is, 
$\mathrm{lk}_{2,1}^a=p_1+\mathrm{lk}_{2,1}^{\widetilde{X_k}}$,  and similarly, 
$\mathrm{lk}_{2,1}^b=q_1+\mathrm{lk}_{2,1}^{\widetilde{\Delta^2}}$. 
By definition, for a braid $c$, the braid $\widetilde{c}$ has a negative (respectively positive) half twist at the place which is a fiber of a point of each arc forming a positive (respectively negative) crossing of $c$; hence, $\mathrm{lk}_{2,1}^{\widetilde{X_k}}=-\mathrm{lk}_{1,2}^{X_k}$ and  
$\mathrm{lk}_{2,1}^{\widetilde{\Delta^2}}=-\mathrm{lk}_{1,2}^{\Delta^2}$. 
thus 
$\mathrm{lk}_{2,1}^a=p_1-\mathrm{lk}_{1,2}^{X_k}$ and 
$\mathrm{lk}_{2,1}^b=q_1-\mathrm{lk}_{1,2}^{\Delta^2}$. 

Further, $\mathrm{lk}_{2,3}^{a}=\mathrm{lk}_{1,2}^{X_k}$ and $\mathrm{lk}_{2,3}^b=\mathrm{lk}_{1,2}^{\Delta^2}$. 
Thus $\mathrm{Tlk}_{1,2,3}(\widetilde{S_k})=-p_1\, \mathrm{lk}_{1,2}^{\Delta^2}+q_1\, \mathrm{lk}_{1,2}^{X_k}$ by (\ref{eq:tlk}). 
Since $\mathrm{lk}_{1,2}^{X_k}$ is the linking number of the closure of $X_k$, $\mathrm{lk}_{1,2}^{X_k}=1$. 
Since $F_1$ (respectively $F_2$) is constructed by one strand (respectively $k$ strands) of $\Delta^2$, $\mathrm{lk}_{1,2}^{\Delta^2}=k$. 
Thus $\mathrm{Tlk}_{1,2,3}(\widetilde{S_k})=-k p_1 +q_1$. 

By the same argument, we have 
$\mathrm{Tl}k_{2,3,4}(\widetilde{S_k})=-p_2\, \mathrm{lk}_{2,1}^{\Delta^2}+q_2 \,\mathrm{lk}_{2,1}^{X_k}$ by (\ref{eq:tlk}), and $\mathrm{lk}_{2,1}^{X_k}=1$. 
Since $\Delta^2$ is a pure braid, we see that $\mathrm{lk}_{2,1}^{\Delta^2}=1$. 
Thus $\mathrm{Tlk}_{2,3,4}(\widetilde{S_k})=-p_2+q_2$. 
\end{proof}

\begin{proof}[Proof of Theorem \ref{mainthm}]
\begin{sloppypar}
Let $k$ and $l$ be positive integers. 
We denote by $F_1$ (respectively $F_2$) the first (respectively second) component of $S_k$, and 
we denote by $F_1^\prime$ (respectively $F_2^\prime$) the first (respectively second) component of $S_l$. 

First we show that for $k\neq l$, there does not exist an orientation-preserving self-diffeomorphism of $\mathbb{R}^4$ carrying $F_1$ to $F_1^\prime$ and $F_2$ to $F_2^\prime$. Assume that there is such a diffeomorphism $f$. 
Let us consider a 2-dimensional braid over $S_k$, denoted by $\widetilde{S_k}^1$, which is presented by a $2$-chart $\Gamma$ on $S_k$ such that $\Gamma \cap F_1$ consists of loops with $I(\mu_1, \Gamma)=2p$ and $I(\lambda_1, \Gamma)=2q$ and $\Gamma \cap F_2=\emptyset$,  where $(\mu_1, \lambda_1)$ is a preferred basis of $H_1(F_1; \mathbb{Z})$. Note that $\widetilde{S_k}^1$ consists of 4 components.    
 
Since $f$ is an orientation-preserving diffeomorphism which carries $F_1$ to $F_1'$, $f |_{F_1}$ is an orientation-preserving diffeomorphism from a torus $F_1$ to a torus $F_1^\prime$. Let $A=\begin{pmatrix} \alpha & \beta \\ \gamma & \delta \end{pmatrix} \in \mathrm{GL}_+(2,  \mathbb{Z})$ be a matrix determined by  
\begin{equation}\label{0323-1}
\begin{pmatrix} \mu_1^\prime \\ \lambda_1^\prime \end{pmatrix}=A\begin{pmatrix} f_*(\mu_1) \\ f_*(\lambda_1) \end{pmatrix},
\end{equation}
where $(\mu_1^\prime, \lambda_1^\prime)$ is a preferred basis of $H_1(F_1'; \mathbb{Z})$. 

Put $\Gamma'=f(\Gamma)$. By $f$, $\widetilde{S_k}^1$ is taken to a 2-dimensional braid over $S_l$, presented by a 2-chart $\Gamma'$ on $S_l$ such that $\Gamma' \cap F_1'$ consists of loops and $\Gamma' \cap F_2'=\emptyset$, which will be denoted by $\widetilde{S_l}^1$. 
We see that $I(f_*(\mu_1), \Gamma')=I(\mu_1, \Gamma)=2p$, and  $I(f_*(\lambda_1), \Gamma')=I(\lambda_1, \Gamma)=2q$. 
Put $p'=I(\mu_1', \Gamma')/2$ and $q'=I(\lambda_1', \Gamma')/2$; note that $p'$ and $q'$ are integers, since $\widetilde{S_l}^1$ consists of $4$ components. 
It follows from (\ref{0323-1}) that 
\begin{equation}\label{0328-1}
\begin{pmatrix} p^\prime \\ q^\prime \end{pmatrix}=A\begin{pmatrix} p \\ q \end{pmatrix}.
\end{equation}
Since the triple linking numbers $\mathrm{Tlk}_{1,2,3}$ for $\widetilde{S_k}^1$ and $\widetilde{S_l}^1$ are the same, Lemma \ref{lemma1} implies that 
\begin{equation}\label{0325-1}
-kp+q=-lp^\prime +q^\prime, 
\end{equation}
hence, it follows from (\ref{0328-1}) that
$
k p-q=(\alpha l-\gamma) p+(\beta l-\delta) q. 
$
Since this equation holds true for any integers $p$ and $q$, 
 \begin{equation}\label{eq:0303-01}
\begin{pmatrix}
k \\ -1
\end{pmatrix}= A^T \begin{pmatrix} l \\ -1 \end{pmatrix},
\end{equation}
where $A^T$ is the transposed matrix of $A$. 

Next we will consider another 2-dimensional braid over $S_k$, denoted by $\widetilde{S_k}^2$, presented by a $2$-chart $\widetilde{\Gamma}$ on $S_k$ such that $\widetilde{\Gamma} \cap F_1=\emptyset$ and $\widetilde{\Gamma} \cap F_2$ consists of loops on $F_2$ and moreover $\widetilde{\Gamma} \cap F_2$ is the preimage by the projection $N(T) \to T$ of a 2-chart $\Gamma$ on the standard torus $T$ consisting of loops with $I(\mathbf{m}, \Gamma)=2p$ and $I(\mathbf{l}, \Gamma)=2q$, where $(\mathbf{m}, \mathbf{l})$ is a preferred basis of $T$. 
Note that $I(\mu_2, \widetilde{\Gamma})=2kp$ and $I(\lambda_2, \widetilde{\Gamma})=2q$, where $(\mu_2, \lambda_2)$ is a preferred basis of $H_1(F_2; \mathbb{Z})$. 

Let $g$ be an orientation-preserving diffeomorphism of $\mathbb{R}^4$ which carries $F_2$ sufficiently close to $F_1$ and $(g|_{F_i})_*=\mathrm{id}_*: H_1(F_i; \mathbb{Z}) \to g_*(H_1(F_i); \mathbb{Z})$ ($i=1,2$). 
Further, we assume that $T$ is sufficiently close to $F_1$. Then $\begin{pmatrix}\mathbf{m}'\\ \mathbf{l}'\end{pmatrix}=A\begin{pmatrix}(f\circ g)_*(\mathbf{m})\\ (f\circ g)_*(\mathbf{l})\end{pmatrix}$, where $(\mathbf{m}', \mathbf{l}')$ is a preferred basis of $T'=(f\circ g)(T)$. 
Put $\Gamma'=(f\circ g) (\Gamma)$. Then we have 
\begin{equation}\label{0323-3}
\begin{pmatrix}I(\mathbf{m}', \Gamma')\\ I(\mathbf{l}', \Gamma')\end{pmatrix}=A\begin{pmatrix}I(\mathbf{m}, \Gamma)\\ I(\mathbf{l}, \Gamma)\end{pmatrix}.
\end{equation}

Put $S'=(f\circ g)(S_k)$. 
The surface link $S'$ is in the form of a 2-dimensional braid over $T'$ of degree $k+1$. 
For the natural projection $\pi': N(T')=(f \circ g)(N(T)) \to T'$ and a meridian $\mathbf{m}'$ and a preferred longitude $\mathbf{l}'$ of $T'$, let us consider $S' \cap \pi'^{-1}(\mathbf{m}')$ and $S' \cap \pi'^{-1}(\mathbf{l}')$, which are closed $(k+1)$-braids in the 3-dimensional solid tori. In the same way of obtaining basis braids, we obtain $(k+1)$-braids from the closed braids by cutting open the solid tori along the 2-disk $\pi'^{-1}(x_0')$, where $x_0'$ is the intersection point of $\mathbf{m}'$ and $\mathbf{l}'$. We denote the braids by $a$ and $b$. Note that here $T'$ is a standard torus, and hence $(a,b)$ are basis braids, but we can apply the same argument if $T'$ is not a standard torus. 
Since $S'$ consists of two components, $a$ and $b$ satisfy one of the three cases as follows.

\begin{enumerate}
\item[(Case 1)] 
The closure of $a$ is a link consisting of two components, and $b$ is a pure braid. 

\item[(Case 2)]
 Each of the closures of $a$ and $b$ is a link consisting of two components. 

\item[(Case 3)]
 The braid $a$ is a pure braid, and the closure of $b$ is a link consisting of two components. 
\end{enumerate}

Put $\widetilde{\Gamma}'= (f \circ g)(\widetilde{\Gamma})$. By $f\circ g$, $\widetilde{S_k}^2$ is taken to a 2-dimensional braid presented by a 2-chart $\widetilde{\Gamma}'$ on $S'$, which will be denoted by $\widetilde{S'}$. 
We denote by $F'$ the component $(f \circ g)(F_2)$ of $S'$, and we denote by $(\mu', \lambda')$ a preferred basis of $H_1(F'; \mathbb{Z})$. 
Since $\widetilde{\Gamma} \cap F_2$  is in the form of the preimage by $N(T) \to T$ of the $2$-chart $\Gamma$ on $T$, $\widetilde{\Gamma}' \cap F'$ is in the form of the preimage by $N(T') \to T'$ of the $2$-chart $\Gamma'$ on $T'$, and hence 
$I(\mu', \widetilde{\Gamma}')=i \cdot I(\mathbf{m}', \Gamma')$ and $I(\lambda', \widetilde{\Gamma}')=j \cdot I(\mathbf{l}', \Gamma')$ for $(i,j)=(k,1)$ for Case 1, $(k,k)$ for Case 2, and $(1,k)$ for Case 3. 
Thus 
\begin{equation}\label{0323-4}
\begin{pmatrix} I(\mu', \widetilde{\Gamma}') \\ I(\lambda', \widetilde{\Gamma}')\end{pmatrix}=B\begin{pmatrix} I(\mathbf{m}^\prime, \Gamma') \\ I(\mathbf{l}^\prime, \Gamma')\end{pmatrix},
\end{equation}
 where $B$ is a diagonal matrix $\mathrm{diag} (i,j)$ such that $(i,j)=(k,1)$ for Case 1, $(k,k)$ for Case 2, and $(1,k)$ for Case 3. 

Put $h=f \circ (f\circ g)^{-1}$. Then $h$ is an orientation-preserving self-diffeomorphism of $\mathbb{R}^4$ which carries $S'$ to $S_l$. In particular, $h$ carries $F'$ to the second component $F_2^\prime$ of $S_l$. 
Let $C=\begin{pmatrix} \alpha^\prime & \beta^\prime \\
\gamma^\prime & \delta^\prime \end{pmatrix} \in \mathrm{GL}_+(2, \mathbb{Z})$ be a matrix determined by  $\begin{pmatrix} \mu_2^\prime \\ \lambda_2^\prime \end{pmatrix}=C \begin{pmatrix} h_*(\mu') \\ h_*(\lambda') \end{pmatrix}$, where $(\mu_2^\prime, \lambda_2^\prime)$ is a preferred basis of $H_1(F_2^\prime; \mathbb{Z})$. Put $\Gamma''=h(\widetilde{\Gamma}')$. Then 
\begin{equation}\label{0323-5}
\begin{pmatrix} I(\mu_2', \Gamma'') \\ I(\lambda_2', \Gamma'') \end{pmatrix}=C\begin{pmatrix} I(\mu', \widetilde{\Gamma}') \\ I(\lambda', \widetilde{\Gamma}')\end{pmatrix}.
\end{equation}

Put $p''=I(\mu_2', \Gamma'')/2$ and $q''=I(\lambda_2', \Gamma'')/2$, which are integers. Since $I(\mathbf{m}, \Gamma)=2p$ and $I(\mathbf{l}, \Gamma)=2q$, 
together with (\ref{0323-3})--(\ref{0323-5}), we have 
\begin{equation}\label{0323-6}
\begin{pmatrix}p''\\ q''\end{pmatrix}=(CBA)\begin{pmatrix} p\\q \end{pmatrix}.
\end{equation}

By the composite diffeomorphism $h \circ f \circ g=f$, $\widetilde{S_k}^2$ is taken to a 2-dimensional braid over $S_l$, which will be denoted by $\widetilde{S_l}^2$. 
Since $\mathrm{Tlk}_{2,3,4}$ are the same for $\widetilde{S_k}^2$ and $\widetilde{S_l}^2$, together with $I(\mu_2, \widetilde{\Gamma})=2kp$ and $I(\lambda_2, \widetilde{\Gamma})=2q$, Lemma \ref{lemma1} implies that  
\begin{equation}\label{0325-2}
-kp+q=-p''+q''. 
\end{equation}
Since this equation holds true for any integers $p$ and $q$, it follows from (\ref{0323-6}) that  
$\begin{pmatrix} k \\ -1 \end{pmatrix}=(CBA)^T \begin{pmatrix}1 \\ -1 \end{pmatrix}$. 
Thus, together with (\ref{eq:0303-01}), $B^T C^T\begin{pmatrix}1 \\ -1 \end{pmatrix}=\begin{pmatrix} l \\ -1 \end{pmatrix}$, hence  
$i(\alpha^\prime-\gamma^\prime)=l$ and $j (\beta^\prime-\delta^\prime)=-1$. 
Let us assume $k>l>0$. For Cases 1 and 2, $k(\alpha^\prime-\gamma^\prime)=l$ from the first equation. This contradicts the assumption that $k>l>0$. For Case 3, the second equation implies that $k(\delta^\prime-\beta^\prime)=1$, which contradicts the assumption that $k>1$. Thus, for $k\neq l$, there does not exist an orientation-preserving self-diffeomorphism of $R^4$ which carries $F_1$ to $F_1'$ and $F_2$ to $F_2'$. 

Next we show that for $k \neq l$, there does not exist an orientation-preserving self-diffeomorphism of $R^4$ which carries $F_1$ to $F_2'$ and $F_2$ to $F_1'$. We will discuss a similar argument as in the former case 
of a diffeomorphism which carries $F_1$ to $F_1'$ and $F_2$ to $F_2'$, using the same notation except where we give notice. 

Assume that there is such a diffeomorphism $f$, and consider $\Gamma$ as in the former case. Then, since $\mathrm{Tlk}_{1,2,3}$ for $\widetilde{S_k}^1$ and $\mathrm{Tlk}_{3,4,1}=\mathrm{Tlk}_{4,3,2}$ (see Lemma \ref{lemma2}) for $\widetilde{S_l}^1$ are the same, and since $\mathrm{Tlk}_{4,3,2}=-\mathrm{Tlk}_{2,3,4}$ \cite{CJKLS}, Lemma \ref{lemma1} implies that instead of (\ref{0325-1}) 
we have
\begin{equation}\label{6-10}
-kp+q=p'-q', 
\end{equation}
where $p'=I(\mu_2', \Gamma')/2$ and $q'=I(\lambda_2', \Gamma')/2$, 
and hence instead of (\ref{eq:0303-01}) we have
\begin{equation}
\begin{pmatrix}\label{0327-1}
k\\-1
\end{pmatrix}
=A^T\begin{pmatrix}
-1\\1
\end{pmatrix}. 
\end{equation}

 Next we will consider another 2-dimensional braid over $S_k$, denoted by $\widetilde{S_k}^2$, presented by the $2$-chart $\widetilde{\Gamma}$ as in the former case.  
Then, by the same argument as in the former case, we have 
(\ref{0323-6}), where $p''=I(\mu_1', \Gamma'')/2$ and $q''=I(\lambda_1', \Gamma'')/2$.

By the composite diffeomorphism $h \circ f \circ g$, $\widetilde{S_k}^2$ is carried to a 2-dimensional braid over $S_l$, which will be denoted by $\widetilde{S_l}^2$. 
Since $\mathrm{Tlk}_{2,3,4}$ for $\widetilde{S_k}^2$ and $\mathrm{Tlk}_{3,1,2}=\mathrm{Tlk}_{3,2,1}$ (see Lemma \ref{lemma2}) for $\widetilde{S_l}^2$ are the same, and since $\mathrm{Tlk}_{3,2,1}=-\mathrm{Tlk}_{1,2,3}$ \cite{CJKLS}, together with $I(\mu_2, \widetilde{\Gamma})=2kp$ and $I(\lambda_2, \widetilde{\Gamma})=2q$, Lemma \ref{lemma1} implies that  
\begin{equation}\label{0325-2}
-kp+q=lp''-q''. 
\end{equation}
Since this equation holds true for any integers $p$ and $q$, it follows from (\ref{0323-6}) that   
$\begin{pmatrix} k \\ -1 \end{pmatrix}=(CBA)^T \begin{pmatrix}-l \\ 1 \end{pmatrix}$. 
Thus, together with (\ref{0327-1}), $B^T C^T\begin{pmatrix}-l \\ 1 \end{pmatrix}=\begin{pmatrix} -1 \\ 1 \end{pmatrix}$, hence  
$i(-l\alpha^\prime+\gamma^\prime)=-1$ and $j (-l\beta^\prime+\delta^\prime)=1$. 
Let us assume $k>l>0$. Since at least one of $i$ and $j$ is $k$ for Cases 1, 2, and 3, these equations contradict the assumption that $k>1$. Thus, for $k \neq l$, there does not exist an orientation-preserving self-diffeomorphism of $R^4$ carries $F_1$ to $F_2'$ and $F_2$ to $F_1'$. Thus $S_k$ and $S_l$ are not equivalent for positive integers $k \neq l$. 
\end{sloppypar}
\end{proof}
 
\section*{Acknowledgments}
The author would like to thank Professors Seiichi Kamada, Shin Satoh, Hiroki Kodama, Takuya Sakasai and the referees for their helpful comments. 
The author was supported by iBMath through the fund for Platform for Dynamic Approaches to Living System from MEXT.

\end{document}